\documentclass[oneside,english]{amsart}
\usepackage[T1]{fontenc}
\usepackage[latin1]{inputenc}
\usepackage{indentfirst}
\usepackage{amsmath}
\usepackage{yhmath} 
\usepackage{amsfonts}
\usepackage{amssymb}
\usepackage{textcomp} 
\usepackage{epic}
\usepackage{verbatim} 

\theoremstyle{definition}
\newtheorem{defni}{{Definition}}
\newtheorem{defn}{{Definition}}[section]
\newtheorem*{defn*}{{Definition}}

\newtheorem*{Q*}{{Question}}

\theoremstyle{plain}
\newtheorem{thm}[defn]{Theorem}
\newtheorem*{thm*}{Theorem}
\newtheorem{cor}[defn]{{Corollary}}
\newtheorem*{cor*}{Corollaire}

\newtheorem{prop}[defn]{{Proposition}}
\newtheorem*{prop*}{{Proposition}}
\newtheorem{lem}[defn]{{Lemma}}
\newtheorem*{claim}{{Claim}}
\newtheorem{fact}[defn]{{Fact}}

\theoremstyle{remark}
\newtheorem{rmq}[defn]{{Remark}}
\newtheorem*{exa}{{Examples}}

\title{On properties of (weakly) small groups}
\author{C\'e{d}ric Milliet}
\address[]{Universit\'e de Lyon, Universit\'e Lyon 1\newline
Institut Camille Jordan, UMR 5208 CNRS\newline
43 boulevard du 11 novembre 1918\newline
69622 Villeurbanne Cedex, France}%
\address[Current address]{Université Galatasaray\newline
Faculté de Sciences et de Lettres\newline
D\'epartement de Math\'ematiques\newline
Çira\u gan Caddesi n°36\newline
34357 Ortaköy, Istamboul, Turquie\newline
}
\email{milliet@math.univ-lyon1.fr}%
\thanks{The results of this paper form part of the author's doctoral dissertation, written in Lyon under the supervision of professor Frank O. Wagner. Many thanks to prof. Poizat for his enlightening remarks on the author's work, and to the anonymous referee for his careful readings and pointing out inaccuracies in the proofs.}

\begin{document}

\begin{abstract}A group is small if it has countably many complete $n$-types over the empty set for each natural number~$n$. More generally, a group $G$ is weakly small if it has countably many complete $1$-types over every finite subset of~$G$. We show here that in a weakly small group, subgroups which are definable with parameters lying in a finitely generated algebraic closure satisfy the descending chain conditions for their traces in any finitely generated algebraic closure. An infinite weakly small group has an infinite abelian subgroup, which may not be definable. A small nilpotent group is the central product of a definable divisible group with a definable one of bounded exponent. In a group with simple theory, any set of pairwise commuting elements is contained in a definable finite-by-abelian subgroup. First corollary : a weakly small group with simple theory has an infinite definable finite-by-abelian subgoup. Secondly, in a group with simple theory, a normal solvable group $A$ of derived length $n$ is contained in an $A$-definable almost solvable group of class $n$.\end{abstract}

\maketitle
A connected group of Morley rank $1$ is abelian \cite[Reineke]{Rei}. Better, in an omega-stable group, a definable connected group of minimal Morley rank is abelian. This implies that every infinite omega-stable group has a definable infinite abelian subgroup \cite[Cherlin]{Che1}. Berline and Lascar generalised this result to superstable groups in \cite{BL}. More recently, Poizat introduced \emph{$d$-minimal structures} (englobing minimal ones) and \emph{structures with finite Cantor rank} (including both $d$-minimal and finite Morley ranked structures). Poizat proved a $d$-minimal group to be abelian-by-finite \cite{dminimal}. He went further showing that an infinite group of finite Cantor rank has a definable abelian infinite subgroup \cite{Cantor}. More generally, we show in this paper that an infinite weakly small group has an infinite abelian subgroup, which may not be definable however.

We then turn to weakly small groups with a simple theory. Recall that an $\aleph_0$-categorical superstable group is abelian-by-finite \cite[Baur, Cherlin and Macintyre]{BCM}. In \cite{W2}, Wagner showed any small stable infinite group to have a definable infinite abelian subgroup of the same cardinality. Later on, Evans and Wagner proved that an $\aleph_0$-categorical supersimple group is finite-by-abelian-by-finite and has finite $SU$-rank \cite{EW}. We shall show that an infinite group the theory of which is small and simple has an infinite definable finite-by-abelian subgroup. However we still do not known whether a stable group must have an infinite abelian subgroup or not.

\begin{defni}A \emph{theory} is \emph{small} if it has countably many $n$-types without parameters for every natural number~$n$. A \emph{structure} is \emph{small} if its theory is so.\end{defni}

Note that smallness is preserved by interpretation, and by adding finitely many parameters to the language. Small theories arise when one wishes to count the number of pairwise non-isomorphic countable models of a complete first order theory in a countable language. If such a theory has fewer than the maximal number of pairwise non-isomorphic models, it is indeed small. Note that $\aleph_0$-categorical theories and omega-stable theories are small.

\begin{defni}\emph{(Belegradek)} A structure is \emph{weakly small} if it has countably many $1$-types over~$a$ for any finite tuple $a$ coming from the structure.\end{defni}

Weakly small structures were introduced by Belegradek to give a common generalisation of small and minimal structures. A weakly small $\aleph_0$-saturated structure is small.

\begin{defni}\emph{(Poizat \cite{dminimal})} An infinite structure is \emph{$d$-minimal} if any of its partitions has no more than $d$ infinite definable subsets.\end{defni}

Provided that its language be countable, a $d$-minimal structure is weakly small as there are at most $d$ non algebraic types over every finite parameter set, and fewer algebraic types than the countably many formulae. Note that weak smallness neither is a property of the theory, nor allows the use of compactness, nor guarantees that the set of $2$-types be countable. It allows arguments using formulae in one free variable only. Those formulae, the parameters of which lie in a fixed finite set, are ranked by the Cantor rank and degree.

\begin{exa}\emph{A non weakly small group.} Let $G$ be the sum over all prime numbers $p$ of cyclic groups of order $p$. For every set of prime numbers $P$, the type saying that "$x$ is $p$-divisible if and only if $p$ is in $P$" is finitely consistent. This produces as many complete types as there are sets of primes, preventing $G$ from being weakly small.

\emph{A non minimal, $d$-minimal group.} Recall that a minimal group is abelian \cite[Reineke]{Rei}, and a $d$-minimal group is abelian-by-finite \cite[Poizat]{dminimal}. Let $M$ be a minimal group, and~$F$ a finite group of order $d$. Any semi-direct product $M\rtimes F$ with a predicate interpreting $M$ will do.

\emph{A non $d$-minimal, non small, weakly small group.} Let $p$ be a prime, and~$G$ the sum over all natural numbers~$n$ of the cyclic groups of order $p^n$. The theory of~$G$ is the the theory of a  $\mathbf Z$-module, and eliminates quantifiers up to positive-prime formulae. So every definable subset of~$G$ is a boolean combination of cosets of subgroups of the form $p^nG$, or $p^nx=0$. This allows only countably many $1$-types over every finite subset, thus $G$ is weakly small. 
But Fact~\ref{gpabel} shows that it is not small.
\end{exa}

\section{The Cantor rank}

Given a structure $M$, a set $A$ of parameters lying inside $M$, and an $A$-definable subset $X$ of~$M$, we define the \emph{Cantor rank of~$X$ over~$A$} by the following induction~:\begin{itemize}\item[] $CB_A(X)\geq 0$ if $X$ is not empty,
\item[] $CB_A(X)\geq \alpha+1$ if there are infinitely many disjoint $A$-definable subsets of~$X$ having Cantor rank over~$A$ at least $\alpha$.
\item[] $CB_A(X)\geq \lambda$ for a limit ordinal $\lambda$, if $CB_A(X)$ is at least $\alpha$ for every $\alpha$ less than $\lambda$.\end{itemize}

If the structure is weakly small and if $A$ is a finite set, this transfinite process eventually stops, and~$X$ has an ordinal Cantor rank over~$A$. 

The \emph{Cantor rank $CB_A(p)$ of a complete $1$-type $p$ in $M$ over~$A$} is the least Cantor rank of the $A$-definable sets implied by $p$. It is also the derivation rank of~$p$ in the topological space $S_1(A)$ (sometimes plus $1$, depending on the definition taken for the Cantor-Bendixson rank).

The \emph{Cantor degree of~$X$ over~$A$} is the greatest natural number $d$ such that there is a partition of~$X$ into~$d$ $A$-definable sets having maximal Cantor rank over~$A$. We shall write $dCB_A(X)$ for this degree. It is also the number of complete types over~$A$ having maximal Cantor rank over~$A$.



For a natural number~$n$, we say that a map is \emph{$n$-to-one} if it is surjective and if the cardinality of its fibres is bounded by~$n$. Definable $n$-to-one maps preserve the Cantor rank, and the degree variations can be bounded by the maximal size of the finite fibres~:

\begin{lem}\label{fibres}Let $X$ and~$Y$ be $A$-definable sets, and~$f$ an $A$-definable map from~$X$ to~$Y$. Then
\begin{enumerate} \item[$(1)$] If $f$ is onto, $CB_A(X)\geq CB_A(Y)$.\
\item[$(2)$] If $f$ has bounded fibres, $CB_A(Y)\geq CB_A(X)$.\
\item[$(3)$] If $f$ is $n$-to-one, then $X$ and~$Y$ have the same Cantor rank over~$A$, and $$dCB_A(Y)\leq dCB_A(X)\leq n\cdot dCB_A(Y)$$
\end{enumerate}\end{lem}

\begin{rmq}The first two points appear for one-to-one maps together with the introduction of Morley's rank \cite[Theorem~2.3]{Mor}. Poizat extends them for $n$-to-one maps in the context of groups with finite Cantor rank \cite[Lemme 1]{Cantor} (independently to the author's work). To the author's knowledge, the result concerning the degree is new.\end{rmq}

\begin{proof}We may add $A$ to the language. For point one, we show inductively that $CB(X)$ is at least $CB(Y)$. If~$CB(Y)\geq \alpha+1$, there are infinitely many disjoint definable sets $Y_0,Y_1\dots$ in $Y$ of rank at least $\alpha$. Their pre-images are disjoint and have rank at least $\alpha$ by induction, so $CB(X)\geq\alpha+1$.

For point two, we show inductively that $CB(Y)$ is at least $CB(X)$. Suppose $CB(X)\geq\alpha +1$. In $X$, there are infinitely many disjoint definable sets $X_0,X_1,\dots$ of rank at least $\alpha$. As the fibres of~$f$ have cardinality at most~$n$ say, for every subset $I$ of~$\mathbf N$ of cardinality~$n+1$, the intersection $\bigcap_{i\in I}f(X_i)$ is empty. Thus there is a subset $J$ of~$\mathbf N$ of maximal finite cardinal with $0$ in $J$ such that~$\bigcap_{i\in J} f(X_i)$ has the same rank as $f(X_0)$. Put $Y_0=\bigcap_{i\in J} f(X_i)$. Iterating, one builds a sequence $Y_0,Y_1,\dots$ of definable sets such that the sets $Y_i$ and~$f(X_i)$ have the same Cantor rank and~$CB(Y_i\cap Y_j)<CB(Y_i)$ for all natural numbers $i\neq j$. Inductively, one may cut off a small ranked subset from every $Y_i$ and assume that they are pairwise disjoint. By induction hypothesis, the rank of every $Y_i$ is at least $\alpha$, so $CB(Y)\geq\alpha +1$.

For the third point, if $Y$ has degree $d$, then there is a partition of~$Y$ in definable sets $Y_1,\dots, Y_d$ with maximal rank. The pre-images of the sets $Y_i$ have maximal rank according to the first two points and form a partition of~$X$, so the degree $dCB(X)$ is at least $dCB(Y)$.

For the converse inequality, let $Y$ have degree $d$, and let $Y_1$ be a subset of~$Y$ of degree $1$. It is enough to show that $f^{-1}(Y_1)$ has degree at most~$n$. Suppose there are $n+1$~disjoint definable subsets $X_0,\dots,X_{n}$ of~$f^{-1}(Y_1)$ with maximal rank. As the fibres of~$f$ have no more that $n$~elements, the intersection $\bigcap_{i=0}^n f(X_i)$ is empty, so there is a proper minimal subset $I$ of~$\{0,\dots,n\}$ such that $\bigcap_{i\in I} f(X_i)$ has the same rank as $Y$. Thus, the intersection of~$\bigcap_{i\in I} f(X_i)$ and~$f(X_i)$ has small rank for every $i$ out of~$I$, and~$dCB(Y_1)$ is at least two, a contradiction.\end{proof}

\begin{rmq}In Lemma~\ref{fibres}(3), to deduce that $X$ and~$Y$ have the same Cantor rank, the fibres of~$f$ must be bounded, and not only finite. Consider for instance $Y$ to be the set of all natural numbers $\mathbf N$ together with the ordering, and~$X$ to be the set of pairs of natural numbers $(x,y)$ so that $y\leq x$. When projecting on the second coordinate, every fibre is infinite, so $CB_\mathbf{N}(X)=2$ ; when projecting on the first coordinate, the fibres are finite, but still $CB_\mathbf{N}(Y)=1$.\end{rmq}

Note that in the proof of Lemma~\ref{fibres}, one can weaken the definability assumption on $f$, and simply assume that the image and pre-image by $f$ of any definable set are definable. For instance, we easily get~:

\begin{lem}\label{aut}Let $M$ be a model, $X$ an $A$-definable subset of~$M$, and~$\sigma$ any automorphism of the structure $M$. Then $$CB_A(X)=CB_{\sigma (A)}(\sigma (X))$$\end{lem}

\begin{defn}Let $M$ be a structure, and $X$ an $acl(\emptyset)$-definable set in $M$. Let $\mathfrak{C}$ be a monster model extending $M$. We consider the finite union of the conjugates of $X(\mathfrak C)$ under the action of $Aut(\mathfrak{C})$. We write $\overline X$ for its intersection with $M$. Similarly, we define $\ring X$ to be the intersection of $M$ with the finite intersection of the conjugates of $X(\mathfrak C)$ under the action of $Aut(\mathfrak{C})$.\end{defn}

Note that neither $\overline X$ nor $\ring X$ depend on the choice of the monster model. $\overline X$ is a $\emptyset$-definable set containing $X$, whereas $\ring X$ is a $\emptyset$-definable subset of $X$.

If $A$ is a subset of~$B$ and~$X$ an $A$-definable set, then $CB_A(X)$ is less than or equal to~$CB_B(X)$. Note that the Cantor rank (respectively degree) of~$X$ over~$A$ or over the definable closure of~$A$ are the same. The Cantor rank over~$A$ also does not change when adding finitely many algebraic parameters to~$A$, and the degree variation can be bounded~:

\begin{lem}\label{acl}Let $X$ be a definable set without parameters, and let $a$ be an algebraic element of degree~$n$ over the empty set. Then
\begin{enumerate} \item[$(1)$]$CB_{a}(X)=CB_\emptyset(X)$
\item[$(2)$] $dCB_\emptyset(X)\leq dCB_{\overline a}(X)\leq n!\cdot dCB_\emptyset(X)$\end{enumerate}\end{lem}

\begin{proof}We assume in the proof that the language is countable. However, this assumption is not necessary (see Remark \ref{RT}). For the first point, the Cantor rank of a set increases when one allows new calculation parameters, so $CB_{a}(X)$ is at least $CB_{\emptyset}(X)$. Conversely, let us show that $CB_\emptyset(X)$ is at least $CB_{a}(X)$. 
Suppose first that~$CB_a(X)=\infty$ holds. Then there must be $2^{\aleph_0}$ types over~$ \overline a$ in $X$. The restriction map from $S(X,\overline a)$ to $S(X,\emptyset)$ is $n!$-to-one. Indeed, if $x$ and $y$ have the same type over $\emptyset$, there is a monster model $\mathfrak{C}$ and an automorphism $\sigma$ of $\mathfrak{C}$ with $y=\sigma (x)$. If $q(x,\overline a)$ is the type of $x$ over $\overline a$, then $q(y,\sigma(\overline a))$ is the type of $y$ over $\overline a$. This shows that there are $2^{\aleph_0}$
types over~$\emptyset$ as well, which yields $CB_\emptyset(X)=\infty$.
So, we may assume that $CB_a(X)$ is an ordinal. Let us suppose that $CB_a(X)=\alpha +1$ and that the result is proved for every $\emptyset$-definable set of~$CB_a$-rank $\alpha$. There are infinitely many disjoint $a$-definable subsets $X_i$ of~$X$, each of one having rank $\alpha$ over~$a$. By Lemma~\ref{aut} and induction hypothesis, for every $i$, the set $X_1$ and a conjugate of~$X_i$ have the same rank (computed over the set $\bar a$ of all conjugates of~$a$). So a conjugate of~$X_1$ intersects only finitely many $X_i$ in a set of maximal rank over~$\bar a$. One can take off these $X_i$, cut off a small ranked subset from the remaining $X_i$ and assume that the conjugates of~$X_1$ do not intersect any $X_i$. Iterating, one may assume that no conjugate of~$X_i$ intersects $X_j$ when $i$ differs from~$j$. 
By Lemma~\ref{aut} and induction hypothesis,
$$CB_a(X_i)=CB_{\bar a}(X_i)=CB_{\bar a}(\overline{X_i})=CB_\emptyset(\overline{X_i})=\alpha$$
As the sets $\overline{X_i}$ are disjoint, $CB_\emptyset(X)\geq \alpha+1$, so the first point is proved.

For the second point, we may assume that $X$ has degree $1$ over the empty set. Suppose that $X$ has degree at least~$n!+1$ over~$\overline a$. Let $X_1$ be an $\overline a$-definable subset of~$X$ with maximal rank over~$\overline a$ and degree $1$. The union $\overline{X_1}$ of its conjugates has degree at most~$n !$ over~$\overline a$, so $\overline{X_1}$ and its complement in $X$ both have maximal rank over~$\overline a$, hence over the empty set, a contradiction.\end{proof}

\begin{defn}We shall call \emph{local Cantor rank of~$X$ over~$acl(a)$} its Cantor rank over any parameter $b$ defining $X$ and having the same algebraic closure as $a$.\end{defn}

\begin{rmq}In Lemma \ref{acl}, if $b$ is another algebraic parameter, one may have $dCB_{a,b}(X)>dCB_a(X)$, so one need not have $CB_\emptyset(X)=CB_{acl(\emptyset)}(X)$. In fact, $CB_{acl(\emptyset)}(X)$ may not even be an ordinal. For instance, consider the unit circle $\mathbf S^1=\{x\in\mathbf C:|x|=1\}$ with a ternary relation $C(a,b,c)$ saying that $b$ lies on the shortest path joining $a$ to~$c$. We add algebraic unary predicates $A_1,A_2,\dots$ to the language, with $A_n=\{x\in \mathbf S^1:x^{2^n}=1\}$ for every natural number~$n$. This structure has $CB_\emptyset$-rank $0$, but infinite $CB_{acl(\emptyset)}$-rank.

\end{rmq}

\begin{rmq}\label{RT}Lemmas \ref{fibres}, \ref{aut} and \ref{acl} are particular cases of a more general topological result. Let $X$ be any Hausdorff topological space, $X'$ his first Cantor derivative, and inductively on ordinals, let $X^{\alpha+1}$ stand for $(X^\alpha)'$. The \emph{Cantor-Bendixson rank of~$X$} is the least ordinal $\beta$ such that $X^\beta$ is empty and $\infty$ if there is no such $\beta$. Let us call a \emph{rough partition} of $X$, any covering of $X$ by open sets having maximal Cantor-Bendixson rank and small ranked pairwise intersections. The \emph{Cantor-Bendixson degree} of $X$ is the supremum cardinal $dCB(X)$ of the rough partitions of~$X$. Without compactness one could have $dCB(X)\geq\omega$. If $X$ was a compact space, one could equivalently define $CB(X)$ (which differs by $1$ from the previous definition) by the following induction~:\begin{itemize}\item[] $CB(X)\geq 0$ if $X$ is not empty,
\item[] $CB(X)\geq \alpha+1$ if there are infinitely many open subsets $O_1,O_2,\dots$ of~$X$ with $CB(O_i)\geq \alpha$ and $CB(O_i\cap O_j)<\alpha$ for all $i\neq j$.
\item[] $CB(X)\geq \lambda$ for a limit ordinal $\lambda$, if $CB(X)\geq\alpha$ for every $\alpha<\lambda$.\end{itemize}

As an analogue of Lemma~\ref{fibres}, replacing a "definable" set by an "open" set, and a "definable" map, by either a "continuous" map or an "open" one, we easily get~:

\begin{lem}\label{top}Let $X$ and~$Y$ be two Hausdorff topological spaces and let $f$ be a map from~$X$ onto~$Y$.\begin{enumerate}
\item If $f$ is open and onto, then $CB(X)\geq CB(Y)$.
\item If $f$ is continuous and has finite fibres, then $CB(Y)\geq CB(X)$.
\item If $f$ is a continuous, open, $n$-to-one, then $CB(X)=CB(Y)$ and $$dCB(Y)\leq dCB(X)\leq n\cdot dCB(Y)$$
\end{enumerate}\end{lem}

To deduce Lemma~\ref{fibres} from Lemma~\ref{top}, we only need to pass from the category of definable sets to the category of topological spaces, and notice that an $A$-definable map $f$ from~$X$ to~$Y$ induces a continuous open map $\tilde{f}$ from the (compact) Hausdorff space of types~$S(X,A)$ to~$S(Y,A)$. Note that in Lemma~\ref{top}(2), the map need only have finite fibres to get preservation of the rank, whereas it needs to have bounded fibres in Lemma~\ref{fibres}(2). Note also that $f$ must have bounded fibres to ensure that $\tilde{f}$ have finite ones. For Lemma~\ref{acl}, consider any continuous equivalence relation $R$ on a Hausdorff topological space $X$, that is a relation such that the canonical map $X\rightarrow X/R$ is open. If every equivalence class of~$R$ has size at most some natural number~$n$, as $X/R$ is Hausdorff and as the map $X\rightarrow X/R$ is also continuous by definition, it follows from Lemma~\ref{top} that $CB(X)=CB(X/R)$ and the inequalities $dCB(X)\leq dCB(X/R)\leq n\cdot dCB(X)$ hold. Let $M$ be any first order structure, $a$ an algebraic parameter of degree $n$, and $\mathfrak C$ a monster model extending $M$. Applied to the space of types over $\overline a$, modulo the equivalence relation "to be conjugated under the action of~$Aut(\mathfrak C)$", the latter yields 
Lemma~\ref{acl}.
\end{rmq}

\section{General facts about weakly small groups}

As an immediate corollary of Lemma~\ref{fibres} we obtain a result of Wagner~:

\begin{cor}\emph{(Wagner \cite{W2})} If $f$ is a definable group homomorphism of a weakly small group $G$, the kernel of which has at most~$n$ elements, then $f(G)$ has index at most~$n$ in $G$.\end{cor}

\begin{proof}Otherwise, one can find a finite tuple $a$ over which at least~$n+1$ cosets of~$f(G)$ are definable, so $G$ has degree over~$a$ at least $(n+1)\cdot dCB_a(f(G))$, a contradiction with Lemma~\ref{fibres}(3).\end{proof}

\begin{cor}\label{conj}In a weakly small group, there are at most $n$~conjugacy classes of elements the centraliser of which has order at most~$n$.\end{cor}

\begin{proof}Otherwise, let us pick $n+1$~conjugacy classes $C_1,\dots, C_{n+1}$ of elements the centraliser of which has order at most~$n$, and choose a finite tuple $a$ over which these classes are definable. According to Lemma~\ref{fibres}, each class $C_i$ has maximal Cantor rank over~$a$ and degree at least $dCB_a(G)/n$, a contradiction.\end{proof}

For any set $X$ definable in an omega-stable group, one can define the stabiliser of~$X$ up to some small Morley ranked set. In a weakly small group, we can define a local stabiliser up to some set of small local Cantor rank, where local means "in a finitely generated algebraic closure". We write $A\Delta B$ for the symmetric difference of two sets $A$ and $B$.

\begin{defn}Let $X$ be a definable set without parameters in a weakly small group $G$, and let $\Gamma$ stand for the algebraic closure of a finite tuple $g$ in $G$. One defines the \emph{local almost stabiliser} of~$X$ in $\Gamma$ to be $$Stab_\Gamma(X)=\{x\in \Gamma: CB_{x,g}(x X\Delta X)<CB_g(X)\}$$
For any subroup $\delta$ of~$\Gamma$, we shall write $Stab_\delta(X)$ for $Stab_\Gamma(X)\cap\delta$.\end{defn}

\begin{cor}\label{subnorm}$Stab_\Gamma(X)$ is a subgroup of~$\Gamma$. If $X$ is invariant by conjugation under elements of~$\Gamma$, then $Stab_\Gamma(X)$ is normal in $\Gamma$.\end{cor}

\begin{proof}Let $a$ and~$b$ be in $Stab_\Gamma(X)$. The sets $X$, $aX$ and~$bX$ have the same types of maximal rank computed over~$g,a,b$, so $CB_{g,a,b}(a X\Delta bX)$ is smaller than $CB_g (X)$. As the rank is preserved under definable bijections, and when adding algebraic parameters, we have$$CB_{g,a,b}(a X\Delta bX)=CB_{g,a,b}(b^{-1}a X\Delta X)=CB_{g,b^{-1}a}(b^{-1}a X\Delta X)$$ so $b^{-1}a$ belongs to~$Stab_\Gamma(X)$.
\end{proof}

Recall that for a definable generic set $X$ of an omega-stable group $G$, the stabiliser of~$X$ has finite index in $G$. For a weakly small group, we have a local version of this fact~:

\begin{prop}\label{fini}Let $G$ be a weakly small group, $g$ a finite tuple of~$G$, and~$X$ a $g$-definable subset of~$X$. If $\delta$ is a subgroup of~$dcl(g)$ and if $X$ has maximal Cantor rank over~$g$, then $Stab_\delta(X)$ has finite index in $\delta$.\end{prop}

\begin{proof}Let $m$ and~$l$ be the degree over $g$ of~$G$ and~$X$ respectively. In $G$, there are $m$ types of maximal rank over $g$ which we call its \emph{generic types over $g$}. Thus, for translates of~$X$ by elements of~$\delta$, there are at most $C_m^l$ choices for their generic types. If one chooses $C_m^l+1$ cosets of~$X$, at least two of them will have the same generic types.\end{proof}

Weakly small groups definable over a finitely generated algebraic closure satisfy a local descending chain condition~:

\begin{lem}\label{findex}Let $G$ be a weakly small group, and~$H_2\leq H_1$ two subgroups of~$G$ definable without parameters.\begin{enumerate}
\item If $H_2\cap acl(\emptyset)$ is properly contained in $H_1\cap acl(\emptyset)$, then either $CB(H_2)<CB(H_1)$, or $dCB(H_2)<dCB(H_1)$.
\item If $H_1$ and~$H_2$ have the same Cantor rank, then $H_2\cap acl(\emptyset)$ has finite index in $H_1\cap acl(\emptyset)$.\end{enumerate}\end{lem}

\begin{proof}If $b$ is an element of~$acl(\emptyset)$ in $H_1\setminus H_2$, the set $\overline{bH_2}$ 
is definable without parameters, and is disjoint from~$H_2$. This proves the first point. If $H_1$ and~$H_2$ have the same Cantor rank, one has $$CB(H_2)=CB_b(H_2)=CB_b(bH_2)=CB_b(\overline{bH_2})=CB(\overline{bH_2})$$ It follows that $CB(\overline{bH_2})$ is maximal in $H_1$, so there must be only finitely many choices for $\overline{bH_2}$, and thus for $bH_2$.\end{proof}

\begin{thm}\label{dccm}In a weakly small group, the trace over~$acl(\emptyset)$ of a descending chain of~$acl(\emptyset)$-definable subgroups becomes stationary after finitely many steps.\end{thm}

\begin{proof}Let $G_1\geq G_2\geq\dots$ be a descending chain of~$acl(\emptyset)$-definable subgroups. According to Lemma~\ref{acl}(1), the local Cantor rank becomes constant after some index~$n$. Then $G_i\cap acl(\emptyset)$ has finite index in $G_n\cap acl(\emptyset)$ for every $i\geq n$ after Lemma~\ref{findex}(2). Let $a$ be some algebraic tuple such that $G_n$ is $a$-definable. By Lemma~\ref{acl}(1), we may add the parameter $a$ in the language and assume without loss of generality that $G_n$ is $\emptyset$-definable. 
The intersection of the $\ring G_i\cap acl(\emptyset)$ when $i\geq n$ is the intersection of finitely many of them by Lemma~\ref{findex}(1)~: it is a subgroup of~$G_n\cap acl(\emptyset)$ of finite index, contained in $G_i$ for every $i\geq n$. The sequence of indexes $[G_n\cap acl(\emptyset):G_i\cap acl(\emptyset)]$ is thus bounded, and bounds the length of the chain $G_1\cap acl(\emptyset)\geq G_2\cap acl(\emptyset)\geq\dots$.\end{proof}

\begin{rmq}We shall call this result \emph{the weakly small chain condition}. Note that Theorem~\ref{dccm} is trivial for an $\aleph_0$-categorical group, and also if one replaces the algebraic closure by the definable closure. \end{rmq}

The weakly small chain condition can be slightly generalised as follows.

Let $X$ be a set, and~$E$, $F$ two equivalence relations on $X$. We write $F\leq E$ if $F$ is finer than $E$, and~$F<E$ whenever $F$ strictly refines $E$. If $Y$ is a subset of~$X$, the relation $E$ may be restricted to~$Y$. We write $E\restriction Y$ this restriction. For every $a$ in $X$, we write $aE$ for the equivalence class of~$a$, and denote the cardinal of the equivalence classes of~$E$ in $X$ by $|X/E|$. If~$F\leq E$, we call \emph{index of~$F$ in $E$ (in $X$)} the cardinal $[E:F]$ defined by $\sup\{|aE/(F\restriction{aE})|:a\in X\}$. Note that $[E:F]=1$ if and only if $E=F$. Moreover, if $G$ is a third equivalence relation on $X$ with $G\leq F\leq E$, then $[E:G]\leq[E:F]\cdot[F:G]$. 

\begin{thm}In a weakly small structure $X$, let $E_0,E_1,\dots$ be a descending chain of~$acl(\emptyset)$-definable equivalence relations on $X$, defined respectively over the algebraic tuples $\bar{e_0},\bar{e_1},\dots$. Assume that for all natural numbers $i$ and all $a,b$ in $X$, the classes $aE_i$ and~$bE_i$ are in $\{a,b,\bar{e_i}\}$-definable bijection. Assume also that for all $i$, $a$ and $b$, the class $a\ring E_i$ and $b\ring E_i$ are in $\{a,b\}$-definable bijection. Then, the trace over~$acl(\emptyset)$ of~$E_i$ becomes constant after finitely many steps.\end{thm}

\begin{proof}We begin by the analogue of Lemma~\ref{findex}.
\begin{lem}\label{findex2}Let $X$ be a weakly small structure, and~$E\leq F$ two $\emptyset$-definable equivalence relation on $X$ such that for every $a,b$ in $X$, the classes $aE$ and~$bE$ (respectively $aF$ and~$bF$) are in $\{a,b\}$-definable bijection.\begin{enumerate} \item If $F\restriction_{acl(\emptyset)}<E\restriction_{acl(\emptyset)}$, then for every $a$ in $acl(\emptyset)$, either $CB_a(aF)<CB_a(aE)$, or $dCB_a(aF)<dCB_a(aE)$.
\item If for some $a$ in $acl(\emptyset)$, the classes $aE$ and~$aF$ have the same local Cantor rank, then the index $[E\restriction{acl(\emptyset)}:F\restriction{acl(\emptyset)}]$ is finite.
\end{enumerate}\end{lem}
\renewcommand{\proofname}{Proof of Lemma~\ref{findex2}}
\begin{proof}As $F\restriction{acl(\emptyset)}<E\restriction{acl(\emptyset)}$, one class $e(E\restriction{acl(\emptyset)})$ must split into at least two disjoint classes $e(F\restriction{acl(\emptyset)})$ and~$a(F\restriction{acl(\emptyset)})$. We may add this new constant $e$ to the language and assume that $eE$ and~$eF$ are $\emptyset$-definable. The sets $\overline{aF}$ and~$eF$ are two $\emptyset$-definable disjoint subset of~$\overline{eE}$ having same local Cantor rank, so the first point is proved. If $aE$ and~$aF$ have the same local Cantor rank, then $\overline{aF}$ and~$\overline{aE}$ do have the same rank over~$\emptyset$, so there are only finitely many choices for $\overline{aF}$, and thus for ${aF}$.
\end{proof}
\renewcommand{\proofname}{Proof}
According to Lemma~\ref{fibres} and Lemma~\ref{findex2}(1), the local Cantor rank of the $E_i$-class of every algebraic element becomes constant after finitely many steps, so we may assume it is constant from~$E_0$. Then the index $[E_0\restriction{acl(\emptyset)}:E_i\restriction{acl(\emptyset)}]$ is finite for every $i$ after Lemma~\ref{findex2}(2). Adding finitely many parameters to the language, we may assume that $E_0$ is $\emptyset$-definable. 
The conjunction of the $\ring E_i\restriction{acl(\emptyset)}$ is the conjunction of finitely many by Lemma~\ref{findex2}(1). It follows that the sequence of indexes $[E_0\restriction{acl(\emptyset)}:E_i\restriction{acl(\emptyset)}]$ must be bounded.\end{proof}




\section{A property of weakly small groups}

\begin{prop}An infinite group whose centre has infinite index, and with only one non-central conjugacy class, is not weakly small.\end{prop}

\begin{rmq}This is the analogue of the stable case \cite[Th\'eor\`eme 3.10]{Poi2} stating that \emph{an infinite group with only one non-trivial conjugacy class is unstable}, which itself comes from the minimal case \cite[Reineke]{Rei}.\end{rmq}

\begin{proof}Note that the group has no second centre. Moding out the centre, we may suppose that the centre is trivial. If there is a non-trivial involution, every element is an involution and the group is abelian, a contradiction. Any non-trivial element $g$ is conjugated to~$g^{-1}$ by some element, say $h$. So $h$ is non-trivial and conjugated to~$h^2$, which equals $h^k$ for some $k$. Write $\delta$ for the definable closure of~$h$ and~$k$. Since $g$ is in $C(h^k)$ and~$gh\neq hg$, the element $h$ belongs to~$(C(C(h))\cap\delta)\setminus(C(C(h^k))\cap\delta)$. It follows that the chain $$C(C(h))\cap \delta>C(C(h^k))\cap \delta>C(C(h^{k^2}))\cap \delta>\cdots$$ is infinite, contradicting the weakly small chain condition \ref{dccm}.\end{proof}

Let $G$ be any group. We say that a subgroup $H$ of~$G$ is \emph{proper} if it is not $G$.

\begin{prop}\label{hm} An infinite non-abelian weakly small group has proper centralisers of cardinality greater than~$n$ for each natural number~$n$.\end{prop}

\begin{proof}For a contradiction, let $G$ be a weakly small counter example with all proper centralisers finite of bounded size~$n$. Note that $G$ has finite exponent, and a finite centre.

\textit{(1) The group $G$ has finitely many conjugacy classes.}

As the centralisers have bounded size, we apply Corollary~\ref{conj}. We may also add a member $a_i$ of each class to the language and assume that every conjugacy class is $\emptyset$-definable.

\textit{(2) We may assume every proper normal subgroup of~$G$ to be central.}

We claim that a normal subgroup must be central or have finite index in $G$~: a normal subgroup is the union of conjugacy classes, hence is $\emptyset$-definable. By Lemma~\ref{fibres}, the conjugacy class of a non central element, $a_1^G$ say, must have maximal Cantor rank over~$\emptyset$. It follows from Lemma~\ref{fibres}(3) that any proper infinite normal subgroup has index at most~$n$. One may replace $G$ by a minimal union $C$ of conjugacy classes (with at least one of them non-central) closed under multiplication~: as the group $C$ has finite index in $G$, every possible non-central proper normal subgroup $H$ in $C$ has finite index in $G$, and would give birth to a subgroup $N$ of~$H$, normal in $G$, and of finite index in $G$, contradicting the minimality of~$C$.

\textit{(3) We may assume that the centre of~$G$ is trivial.}

Should $G/Z(G)$ be abelian, $G/Z(G)$ would be be finite, as $G$ has only finitely many conjugacy classes. This is not possible as $G$ is infinite. It follows that the second centre $Z_2(G)$ of~$G$ is a proper normal subgroup in $G$. By $(2)$, one has $Z_2(G)=Z(G)$. Moding out by the centre (which preserves weak smallness as well as the assumption that the centralisers have bounded size), we may assume that the centre of~$G$ is trivial.

\textit{(4) The group $G$ is not locally finite.}

Assume that $G$ be locally finite. Since it has finite exponent, there is a prime number $p$ such that for every natural number~$n$, there is a finite subgroup $H$ of~$G$ whose cardinality is divisible by $p^n$. Then $H$ has Sylow subgroup $S$ of cardinality at least $p^n$. But $S$ has a non-trivial centre, the centraliser of any element of which contains the whole Sylow, a contradiction. 
Thus, one can consider a finitely generated infinite algebraic closure $\Gamma$.

\textit{(5) The group $\Gamma$ has finitely many conjugacy classes.}

Any $x$ in $\Gamma$ can be written $a_i^y$. As $C(a_i)$ is finite, $y$ is algebraic over~$a_i$ and~$x$.

\textit{(6) One may assume the proper normal subgroups of~$\Gamma$ to be trivial.}

By (2) and (3), no proper union of conjugacy classes $C_1,\dots,C_m$ (in the sense of~$G$) is closed under multiplication. We may add finitely many parameters witnessing this fact to the language.

\textit{(7) For every conjugacy class $a^G$, the group $Stab_{\Gamma} (a^G)$ equals $\Gamma$.}

The local stabiliser of~$a^G$ in $\Gamma$ is a normal subgroup of~$\Gamma$ by Corollary~\ref{subnorm}. It must be non-trivial according to Proposition~\ref{fini}, hence equals $\Gamma$ by $(6)$.

\textit{(8) $G$ has only one non-central conjugacy class.}

We use an argument of Poizat in~\cite{dminimal}, which we shall call \emph{Poizat's symmetry argument}. Let $a=a_i$ and~$b=a_j$ be representatives of any two non-trivial conjugacy classes (in particular, $a,b$ are in $\Gamma$). For every conjugate $xbx^{-1}$ of~$b$ except a set of small Cantor rank over~$a$ and~$b$, the elements $axbx^{-1}$ and~$b$ are conjugates. As a surjection with bounded fibres preserves the rank, for all $x$ except a set of small rank, $axbx^{-1}$ and~$b$ are conjugates. Symmetrically, for all $x$ except a set of small rank, $x^{-1}axb$ and~$a$ are conjugates~: one can find some $x$ such that $axbx^{-1}$ and~$x^{-1}axb$ are conjugated respectively to~$b$ and~$a$. Thus, $b$ and~$a$ lie in the same conjugacy class.

\textit{(9) Final contradiction.}

$G$ is an infinite group with bounded exponent and only one non-trivial conjugacy class. Such a group does not exist \cite[Reineke]{Rei, dminimal}. For instance, as a group of exponent $2$ is abelian, the group should have exponent a prime $p\neq 2$. If~$x\neq 1$, the elements $x$ and~$x^{-1}$ would be conjugated under some element $y$ of order $2$ modulo the centraliser of~$x$, which prevents the group from having exponent $p$.\end{proof}

\begin{thm}\label{thm1}A small infinite $\aleph_0$-saturated group has an infinite abelian subgroup.\end{thm}

\begin{proof}By Proposition~\ref{hm} and saturation, such a group is either abelian, or has an infinite proper centraliser. Iterating, one either ends on an infinite abelian centraliser after finitely many steps or builds an infinite chain of pairwise commuting elements. These elements generate an infinite abelian subgroup.\end{proof}

Appealing to Hall-Kulatilaka-Kargapolov, who use Feit-Thomson's Theorem, one can say much more, and manage without the compactness theorem. Recall 

\begin{fact}[Hall-Kulatilaka-Kargapolov \cite{HK}]An infinite locally finite group has an infinite abelian subgroup.\end{fact}

\begin{thm}\label{thm2}A weakly small infinite group has an infinite abelian subgroup.\end{thm}

\begin{proof}We need just show that any weakly small infinite group is either abelian or has an infinite proper centraliser~: if this is the case, iterating, one either gets an infinite abelian centraliser or builds an infinite chain of pairwise commuting elements.

So let $G$ be a non abelian counter-example. Every non central element of~$G$ has finite centraliser, and~$G$ has a finite centre. 
The group $G$ cannot have an infinite abelian subgroup. According to Hall-Kulatilaka-Kargapolov, $G$ is not locally finite. By Lemma~\ref{fibres}(3), an infinite finitely generated subgroup $\gamma$ splits into finitely many conjugacy classes (in the sense of~$G$). By Lemma~\ref{fibres}, these classes have maximal Cantor rank over~$\gamma$. By Proposition~\ref{fini}, the almost stabiliser of every such class is a normal subgroup of finite index in $\gamma$. After Poizat's symmetry argument, the intersection of almost stabilisers of all conjugacy classes meeting $\gamma$ consists of a (finite) central subgroup $Z$ together with $C_\gamma\cap\gamma$, where $C_\gamma$ is a conjugacy class in $G$. It is easy to see that $C_\gamma$ is the same for all finitely generated infinite subgroups $\gamma$, so we can denote this unique conjugacy class by $C$. We conclude that $C\cup Z$ is a subgroup of~$G$. 
Replacing $G$ by the later, we are back to the case where all proper centralisers have bounded size, a contradiction with Proposition~\ref{hm}.\end{proof}

\begin{rmq}The initial proof of Theorem~\ref{thm1} used Hall-Kulatilaka-Kargapolov. The author is grateful to Poizat who adapted the proof to a weakly small group and made clarifying remarks.\end{rmq}

\begin{rmq}One cannot expect the infinite abelian group to be definable, as Plotkin found infinite $\aleph_0$-categorical groups without infinite definable abelian subgroups~\cite{Plo}.\end{rmq}

\section{Small nilpotent groups}

We now switch to small nilpotent groups. Let us first recall that the structure of small abelian pure groups is already known~:

\begin{fact}\label{gpabel}\emph{(Wagner \cite{WQ})} A small abelian group is the direct sum of a definable divisible group with one of bounded exponent.\end{fact}

\begin{rmq}The group of bounded exponent need not be definable, but it is contained in a definable group of bounded exponent.\end{rmq}

\begin{rmq}\label{PB}Since Pr\"ufer and Baer, one knows that a divisible abelian group is isomorphic to direct sums of copies of~$\mathbf{Q}$ and Pr\"ufer groups, whereas an abelian group of bounded exponent is isomorphic to a direct sum of cyclic groups \cite{Fuc}. It follows that the theory of a small pure group has countably many denumerable pairwise non-isomorphic models~; thus, Vaught's conjecture holds for the theory of a pure abelian group. More generally, Vaught's conjecture holds for every complete first order theory of module over a countable Dedekind ring (and thus for a module over~$\mathbf{Z}$), as well as for several classes of modules over countable rings \cite[Puninskaya]{Pun}.\end{rmq}

\begin{rmq}Fact~\ref{gpabel} does not hold for a weakly small abelian group~: consider the sum over~$n$ of cyclic groups of order $p^n$. But one may say~:\end{rmq}

\begin{prop}In a weakly small abelian group, for every natural number~$n$, any element is the sum of an $n$-divisible element with one of finite order.\end{prop}

\begin{proof}For a contradiction, let us suppose that there be an element $x$ and a natural number $n$ such that $xz\notin G^n$ for any $z$ having finite order. If there is some $y$ in $G$ and some natural number $k$ such that $x^{kn}=y^{(k+1)n}$, this yields $x=y^n(y^{-n}x)$ with $(y^{-n}x)^{kn}=1$, a contradiction. Then, for every natural number $k$, one has $x^{kn}\in G^{kn}\setminus G^{(k+1)n}$. This implies that the chain $G\cap acl(x)>G^n\cap acl(x)>G^{2n}\cap acl(x)>\cdots$ is strictly decreasing and contradicts the weakly small chain condition.\end{proof}

In an abelian group, every divisible group is a direct summand \cite[Theorem~1]{Bae}. This may not be true for a central divisible subgroup of an arbitrary group, even if the ambient group is nilpotent. For instance, consider the subgroup of~$GL_3(\mathbf{C})$ the elements of which are upper triangular matrices with $1$ entries on the main diagonal~; it is a nilpotent group whose centre $Z$ is divisible, isomorphic to~$\mathbf{C}^\times$, but $Z$ is no direct summand. However, we claim the following~:

\begin{prop}\label{Ba2}Let $G$ be a group, and~$D$ a divisible subgroup of the centre. There exists a subset $A$ of~$G$, invariant under conjugation and containing every power of its elements, with in addition $$G=D\cdot A\hspace{.5cm}\text{and}\hspace{.5cm} D\cap A=\{1\}$$\end{prop}

\begin{proof}If $A_1\subset A_2\subset\cdots$ is an increasing chain of subsets each of which contains all its powers and such that $A_i\cap D$ is trivial, then $\bigcup A_i$ still contains all its powers and~$\bigcup A_i\cap D$ is trivial too. By Zorn's~Lemma there is a maximal subset $A$ with these properties. We show that $D\cdot A$ equals $G$. Otherwise, there exists an $x$ not in $D\cdot A$. By maximality of~$A$, there is a natural number~$n$ greater than $1$, and some $d$ in $D$ so that $x^n$ equals $d$. We may choose $n$ minimal with this property. Let $e$ be an $n$th root of~$d^{-1}$ in $D$, and let $y$ equal $xe$. Then $y^n$ equals one, and~$y$ is not in $D\cdot A$. But the set of powers of~$y$ intersects $D$ by maximality of~$A$~: there is some natural number $m<n$ such that $y^{m}$ lie in $D$, and so does $x^m$, a contradiction with the choice of $n$.\end{proof}

In \cite[Nesin]{Nes}, it is shown that an omega-stable nilpotent group is the central product of a definable group with one of bounded exponent. We show that this also holds for a small nilpotent group. Recall that a group $G$ is \emph{the central product of two of its normal subgroups}, if it is the product of these subgroups and if moreover their intersection lies in the centre of~$G$. For a group $G$ and a subset $A$ of~$G$, we shall write $A^n$ for the set of the $n$th-powers of~$A$, and~$G'$ for the derived subgroup of~$G$. The following algebraic facts about nilpotent groups can be found in \cite[Chapter~1]{BN}.

\begin{fact}\label{com}In a nilpotent group, any divisible subgroup commutes with elements of finite order.\end{fact}


\begin{fact}\label{wagcon}Let $G$ be a nilpotent group of nilpotent class~$c$. If $G/G'$ has exponent~$n$, the exponent of~$G$ is a natural number dividing~$n^c$.\end{fact}

\begin{prop}\label{p1}Let $G$ be a nilpotent small group, and~$D$ a divisible subgroup containing $G^n$ for some non-zero natural number $n$. Then $G$ equals the product $D\cdot F$ where the group $F$ has bounded exponent.\end{prop}

\begin{proof}Note that since $D$ is divisible and $G^n\subset D$, we get $G^n= D$. By induction on the nilpotency class of~$G$. If $G$ is abelian, Baer's Theorem~\cite{Bae} concludes. Suppose that the result holds for any small nilpotent group of class~$c$, and that $G$ is nilpotent of class~$c+1$, and let $Z(G)$ be the centre of~$G$. The group $G/Z(G)$ is nilpotent of class~$c$. The quotient $(D\cdot Z(G))/Z(G)$ is a divisible subgroup and contains $\big(G/Z(G)\big)^n$. By induction hypothesis, $G/Z(G)$ equals the product $\big(D\cdot Z(G)/Z(G)\big)\cdot \big(C/Z(G)\big)$ with $C/Z(G)$ of finite exponent, say $m$. On the other hand, the centre is the sum of a divisible subgroup $D_0$ with a subgroup $F_0$ of finite exponent, say $l$. So $C^{lm}$ is included in $D_0$. By Proposition~\ref{Ba2}, there is some set $A$ closed under power operation, such that $C=D_0\cdot A$ and~$D_0\cap A=\{1\}$~; but $A^{lm}$ is included in $D_0\cap A$, so $A$ has finite exponent, and $$G=D\cdot Z(G)\cdot D_0\cdot A=(D\cdot D_0)\cdot (F_0\cdot A)$$ Note that since we have $D_0\subset G^n=D$, we get $G=D\cdot B$ where $B$ is a set having finite exponent. Let $F$ be the group generated by $B$. The abelian group $F/F'$is generated by $(B\cdot F')/F'$ and has bounded exponent. Fact~\ref{wagcon} implies that $F$ has bounded exponent.
\end{proof}

\begin{thm}A small nilpotent group is the central product of a definable divisible group with a definable one of bounded exponent.\end{thm}

\begin{proof}If $G$ is a small abelian group, it is the direct product of a divisible definable group $D$ and of one group $F$ of finite exponent $n$ by Fact~\ref{gpabel}. So it is the product of $D$ and the definable group of every elements of order $n$. By induction on the nilpotency class, if $G$ is nilpotent of class~$c+1$, then $G/Z(G)$ is the central sum of some divisible definable normal subgroup $A/Z(G)$ and some group $B/Z(G)$ of finite exponent~$n$. Besides, $Z(G)$ equals $D_0\oplus F_0$ where $F_0$ has exponent $m$ and~$D_0$ is definable and divisible. We write $D$ for $A^{2m}\cdot D_0$.\begin{claim}$D$ is a definable divisible normal subgroup of~$G$.\end{claim}

\renewcommand{\proofname}{Proof of Claim}
\begin{proof}Let $x$ be an element in $A$ and $q$ a natural number. As As $A/Z(G)$ is divisible there is some $y$ in $A$ with $x^{-1}y^q$ in $Z(G)$. Then $x^{-2m}(y^{2m})^{q}$ is in $D_0$. As $D_0$ is central and divisible, this proves that $A^{2m}\cdot D_0$ is a divisible part. Let us show that it is a group. Let $a,b$ be in $A$. As $A/Z(G)$ is normal in $G/Z(G)$, there exists a central element $z$ such that $ab=baz$. Moreover, we have $z=d_0f_0$ for some $d_0$ in $D_0$ and~$f_0^m=1$. We obtain $$a^{2m}b^{2m}=(ab)^{2m}z^{(2m-1)+(2m-2)+\cdots+1}=(ab)^{2m}z^{m(2m-1)}=(ab)^{2m}d_0^{m(2m-1)}$$
$A$ is a subgroup of $G$ so $ab$ is in $A$, and $a^{2m}b^{2m}$ belongs to~$A^{2m}\cdot D_0$. A similar argument shows that $D$ is normal in $G$.\end{proof}
\renewcommand{\proofname}{Proof}
By Fact~\ref{com}, the set $G^{2mn}$ is included in $D$, so we may apply Proposition~\ref{p1}~: there is a group $B$ of bounded exponent $p$ such that $G=D\cdot B$. We may assume $B$ to be definable and normal by replacing it with the set $\{x\in G:x^p=1\}$ (the fact that $\{x\in G:x^p=1\}$ is a normal subgroup of~$G$ follows from Fact~\ref{com}).\end{proof}

\section{Groups in a small and simple theory}

We shall not define what a simple structure is, but refer the interested reader to~\cite[Wagner]{WS}. We just recall the uniform descending chain condition up to finite index in a group with simple theory~:

\begin{fact}\label{dccs}\emph{(Wagner \cite[Theorem~4.2.12]{WS})} In a group with simple theory, a descending chain of intersections of a family $H_1,H_2\dots$ of subgroups defined respectively by formulae $f(x,a_1),f(x,a_2)\dots$ where $f(x,y)$ is a fixed formula, becomes stationary after finitely many steps, up to finite index.\end{fact}

Recall that two subgroups of a given group are \emph{commensurable} if the index of their intersection is finite in both of them.

\begin{fact}\label{Sch}\emph{(Schlichting \cite{Sch, WS})} Let $G$ be a group and~$\mathfrak{H}$ a family of uniformly commensurable subgroups. There is a subgroup $N$ of~$G$ commensurable with members of~$\mathfrak{H}$ and invariant under the action of the automorphisms group of~$G$ stabilising the family $\mathfrak{H}$ setwise. The inclusions $\bigcap_{H\in\mathfrak{H}}\subset N\subset \mathfrak{H}^4$ hold. Moreover, $N$ is a finite extension of a finite intersection of elements in $\mathfrak H$. In particular, if $\mathfrak{H}$ consists of definable groups then $N$ is also definable.\end{fact}


We now turn to small simple groups. The first step towards the existence of a definable finite-by-abelian subgroup is to appeal to Theorem~\ref{thm1}. Note that in a stable group, every set of pairwise commuting elements is trivially contained in a definable abelian subgroup. Shelah showed that this also holds in a dependent group \cite{She}. The second step is the following~:

\begin{prop}\label{absimp}In a group with simple theory, every abelian subgroup $A$ is contained in an $A$-definable finite-by-abelian subgroup.\end{prop}

\begin{proof}Let $G$ be this group and $\mathfrak{C}$ a sufficiently saturated elementary extension of $G$. We work inside $\mathfrak{C}$. By Fact~\ref{dccs}, there exists a finite intersection $H$ of centralisers of elements in $A$ such that $H$ is minimal up to finite index. The group $H$ contains $A$, and the centraliser of every element in $A$ has finite index in $H$. Consider the \emph{almost centre} $Z^*(H)$ of~$H$ consisting of elements in $H$ the centraliser of which has finite index in $H$. We claim that $Z^*(H)$ is a definable group. It is a subgroup containing $A$. According to \cite[Lemma 4.1.15]{WS}, a definable subgroup $B$ of~$\mathfrak{C}$ has finite index in $\mathfrak{C}$ if and only if the equality $D_\mathfrak{C}(B,\varphi,k)= D_\mathfrak{C}(\mathfrak{C},\varphi,k)$ holds for every formula $\varphi$ and natural number $k$. So we have the following equality $$Z^*(H)=\{h\in H:D_\mathfrak{C}(C_H(h),\varphi,k)\geq D_\mathfrak{C}(H,\varphi,k),\varphi \ \text{formula},k\ \text{natural number}\}$$ Recall that for a partial type $\pi(x,A)$, the sentence "$D_\mathfrak{C}(\pi(x,A),\varphi,k)\geq n$" is a type-definable condition on $A$ as stated in~\cite[Remark 4.1.5]{WS}, so the group $Z^*(H)$ is type-definable. By compactness and saturation, centralisers of elements in $Z^*(H)$ have bounded index in $H$, and conjugacy classes in $Z^*(H)$ are finite of bounded size. The first observation implies that $Z^*(H)$ is definable, and the second one together with \cite[Theorem~3.1]{Neu} show that the derived subgroup of~$Z^*(H)$ is finite. Note that $H$ and $Z^*(H)$ are $A$-definable, hence $Z^*(H)$ computed in $G$ fulfills our purpose.\end{proof}

\begin{cor}\label{SS}A weakly small infinite group the theory of which is simple has an infinite definable finite-by-abelian subgroup.\end{cor}

\begin{proof}Follows from Theorem~\ref{thm2} and Proposition~\ref{absimp}.\end{proof}

\begin{rmq}Corollary~\ref{SS} states the best possible result as there are $\aleph_0$-categorical simple groups without infinite abelian definable subgroups. For instance, infinite extra-special groups of exponent $p$ are $\aleph_0$-categorical \cite[Felgner]{Fel}, and supersimple of~$SU$-rank $1$ as they can be interpreted in an infinite dimensional vector space over~$\mathbf{F}_p$ endowed with a non degenerate skew-symmetric bilinear form. They have no infinite definable abelian subgroup by \cite[Plotkin]{Plo}.\end{rmq}

\begin{cor}A weakly small supersimple group of~$SU$-rank $1$ is finite-by-abelian-by-finite.\end{cor}

As noticed by Aldama in his thesis \cite{Ald}, Shelah's result concerning abelian subsets of a dependent group extends to a nilpotent subset of a dependent group. Actually Aldama also shows that in a dependent group $G$ any solvable group $A$ is surrounded by a definable solvable group of same derived length, provided that $A$ be normal in~$G$. We are interested in analogues of these results in the context of a group with simple theory. We propose the following definition~:

\begin{defn}A group $G$ is \emph{almost solvable of class~$n$} for some natural number~$n$, if there is a decreasing sequence of subgroups $G_i$ such that $$G_0=G\trianglerighteq G'\trianglerighteq G_1\trianglerighteq G'_1\trianglerighteq \dots \trianglerighteq G'_n \trianglerighteq G_{n+1}=\{1\}$$ and such that the index $[G'_i:G_{i+1}]$ is finite, with $G_i$ being normal in $G$ for all $i$. We call the sequence $G_i$ an \emph{almost derived series}.\end{defn}

An almost solvable group of class~zero is a finite-by-abelian group. We may write $H\underset{f}\leq G$ to mean that $H \leq G$ and~$[G:H]$ is finite.



\begin{cor}In a group with simple theory, let $A$ be a solvable subgroup of derived length~$n$. If $A$ is a normal subgroup, There is an $A$-definable almost solvable group of class~$n$ containing $A$ such that the members of the almost derived series are $A$-definable.\end{cor}

\begin{proof}Let us show it by induction on the derived length~$n$. Without loss of generality, we may work in a monster model $\mathfrak C$ extending the ambient group. When $n$~equals~$0$, this is Proposition~\ref{absimp}. Suppose that the result holds until~$n-1$. By induction hypothesis, there is an $A$-definable almost solvable group $G$ of derived length~$n-1$ containing $A'$ with an $A$-definable almost derived series, meaning that there are $A$-definable normal subgroups $G_0,\dots, G_n$ of~$G$ such that $$G=G_0\trianglerighteq G'\underset{f}\trianglerighteq G_1\trianglerighteq\cdots\trianglerighteq G_{n-1}'\underset{f}\trianglerighteq G_n=\{1\}$$ We shall now use an argument of Wagner in \cite{KSW}. Put $P=G_0\times G_1\cdots\times G_n\leq \mathfrak C\times\cdots\times \mathfrak C$. By Fact~\ref{dccs} there is a finite intersection $H$ of $A$-conjugates of~$P$ which is minimal up to finite index. Let us write $\mathfrak{H}$ for the set of $A$-conjugates of~$H$. We claim that the elements of $\mathfrak{H}$ are uniformly commensurable. To see that, we consider the \emph{almost normaliser} $\{g\in \mathfrak C :\textit{ $H^g$ and $H$ are commensurable}\}$ of $H$ in $\mathfrak C$. We write it $N^*_{\mathfrak C}(H)$. By~\cite[Lemma 4.1.15]{WS}, we have~: $$N^*_{\mathfrak C}(H)=\{g\in \mathfrak C:D_\mathfrak{C}(H\cap H^g,\varphi,k)\geq D_\mathfrak{C}(H,\varphi,k),\varphi \ \text{formula},k\ \text{natural number}\}$$ It follows from~\cite[Remark 4.1.5]{WS} that $N^*_{\mathfrak C}(H)$ is an $A$-type-definable group. By compactness and saturation, two $N^*_{\mathfrak C}(H)$-conjugates of $H$ are uniformly commensurable. $N^*_{\mathfrak C}(H)$ is in fact a definable group. As $N^*_{\mathfrak C}(H)$ contains $A$, the elements of $\mathfrak{H}$ are uniformly commensurable. We may now apply Fact~\ref{Sch}, and be able to find an $A$-definable group $P_A\leq\mathfrak C\times\cdots\times \mathfrak C$ commensurable with $H$ and invariant by conjugation under elements of~$A$. We enumerate the coordinates of $P_A$ from $0$ to $n$, we write $\pi_k$ the projection on the $k$th coordinate and put $P_k=\pi_k(P_A)$ for every $k$ in $\{0,\dots,n\}$. 
As $P_A$ is a finite extension of a finite intersection $I$ of elements in $\mathfrak H$, the group $P_0$ contains $A'$. We still have $$P_0\trianglerighteq P_0'\underset{f}\trianglerighteq P_1\trianglerighteq\cdots\trianglerighteq P_{n-1}'\underset{f}\trianglerighteq P_n=\{1\}$$ 
We claim that every $P_k$ is again normal in $P_0$. 
As $G_0$ normalises $P$, the group $\pi_0(H)$ normalises $H$. As $A$ is a normal subgroup of $\mathfrak{C}$, it follows that conjugations by elements from $\pi_0(H)$ stabilise $\mathfrak{H}$ setwise. By Fact~\ref{Sch}, $\pi_0(H)$ normalises $P_A$. Let $a$ be in $A$. Similarly, $\pi_0(H^a)$ normalises $P_A$. As $P_0\subset \pi_0(\mathfrak{H})^4$ by Fact~\ref{Sch}, $P_0$ normalises $P_A$ hence $P_k$ for every $k$. 
Now $P_0$ contains $A'$ so the the group $P_0A/P_0$ is abelian. According to Proposition~\ref{absimp}, there is an $A$-definable group $M$ such that $$P_0A/P_0\leq M/P_0\leq \big(\bigcap_{i=0}^nN_G(P_i)\big)/P_0$$ where $M'/P_0$ is finite. Thus, $M$ is the desired almost solvable group of derived length~$n$.\end{proof}

\begin{Q*}We may define a group $G$ to be \emph{almost nilpotent of class~$n$} if there exists a decreasing sequence of subgroups $G_i$ such that $$G_0=G\trianglerighteq [G,G_0]\trianglerighteq G_1\trianglerighteq [G,G_1]\trianglerighteq G_2\trianglerighteq \dots \trianglerighteq [G,G_n]\trianglerighteq G_{n+1}=\{1\}$$ and such that the index $[[G,G_i]:G_{i+1}]$ is finite and~$G_i$ is normal in $G$ for all $i$. In a group with simple theory, is any nilpotent subgroup of class~$n$ contained in a definable almost nilpotent group of class~$n$?\end{Q*}

\end{document}